\documentclass[10pt,draft]{article}
\usepackage{amsmath,amscd}
\usepackage{amssymb,latexsym,amsthm}
\usepackage{color}
\usepackage[spanish,english]{babel}
\usepackage{amssymb}
\usepackage{color}
\usepackage{amsmath,amsthm,amscd}
\usepackage[latin1]{inputenc}
\usepackage{lscape}
\usepackage{fancyhdr}
\usepackage{amsfonts}
\usepackage{pb-diagram}
\usepackage{hyperref}
\usepackage{float}
\numberwithin{equation}{section}
\newtheorem{theorem}{Theorem}[section]

\newtheorem{definition}[theorem]{Definition}
\newtheorem{proposition}[theorem]{Proposition}

\theoremstyle{definition}

\newtheorem{remark}[theorem]{Remark}

\title{\textbf{Gelfand-Kirillov dimension for rings}}
\author{Oswaldo Lezama\\
\texttt{jolezamas@unal.edu.co}
\\Helbert Venegas\\
\texttt{hjvenegasr@unal.edu.co}
\\ Seminario de Álgebra Constructiva - SAC$^2$\\ Departamento de Matemáticas\\ Universidad Nacional de
Colombia, Sede Bogotá}
\date{}
\begin{document}
\maketitle
\begin{abstract}
\noindent The classical Gelfand-Kirillov dimension for algebras
over fields has been extended recently by J. Bell and J.J Zhang to
algebras over commutative domains. However, the behavior of this
new notion has not been enough investigated for the principal
algebraic constructions as polynomial rings, matrix rings,
localizations, filtered-graded rings, skew $PBW$ extensions, etc.
In this paper we present complete proofs of the computation of
this more general dimension for the mentioned algebraic
constructions for algebras over commutative domains. The
Gelfand-Kirillov dimension for modules and the Gelfand-Kirillov
transcendence degree will be also considered. The obtained results
can be applied in particular to algebras over the ring of
integers, i.e, to arbitrary rings.

\bigskip


\bigskip

\end{abstract}


\section{Introduction}

Let $R$ be a commutative domain and $Q$ be the field of fractions
of $R$. Let $B$ be a $R$-algebra, then $Q\otimes B$ is a
$Q$-algebra and its classical Gelfand-Kirillov dimension is
denoted by ${\rm GKdim} (Q\otimes B)$ (see \cite{Krause}). Recall
that if $M$ is a finitely generated $R$-module, then the
\textit{rank} of $M$ is defined by
\begin{center}
${\rm rank}M:=\dim_Q(Q\otimes_R M)<\infty$.
\end{center}
From now on in this paper all tensors are over $R$.

\begin{definition}[\cite{BellZhang}, Section 1]\label{definitionGKforrings}
Let $R$ be a commutative domain and $B$ be a $R$-algebra. The
Gelfand-Kirillov dimension of $B$ is defined to be
\begin{equation}
{\rm GKdim}(B):=\sup_{V}\overline{\lim_{n\to \infty}}\log_n{\rm
rank} V^n,
\end{equation}
where $V$ varies over all frames of $B$ and $V^n:=\, _R\langle
v_1\cdots v_n\mid v_i\in V, 1\leq i\leq n\rangle$; a frame of $B$
is a finitely generated $R$-submodule of $B$ containing $1$. A
frame $V$ generates $B$ if $B$ is generated by $V$ as $R$-algebra.
\end{definition}

\begin{proposition}[\cite{BellZhang}, Lemma 3.1]\label{proposition17.3.8}
Let $B$ be a $R$-algebra. Then,
\begin{center}
${\rm GKdim}(B)={\rm GKdim}(Q\otimes B)$.
\end{center}
\end{proposition}
\begin{proof}
Let $V$ be a frame of $B$, then $Q\otimes V$ is a frame of the
$Q$-algebra $Q\otimes B$. In fact, if $V=\,_R\langle
v_1,\dots,v_m\rangle$, then $Q\otimes V=\,_Q\langle 1\otimes
v_1,\dots,1\otimes v_m \rangle$ and $1\otimes 1\in Q\otimes V$.
Observe that for every $n\geq 0$,
\begin{center}
$Q\otimes V^n=(Q\otimes V)^n$,
\end{center}
hence
\begin{center}
${\rm GKdim}(B)=\sup_{V}\overline{\lim_{n\to \infty}}\log_n{\rm
rank} V^n=\sup_{V}\overline{\lim_{n\to
\infty}}\log_n(\dim_Q(Q\otimes V^n))=\sup_{V}\overline{\lim_{n\to
\infty}}\log_n(\dim_Q(Q\otimes V)^n)\leq {\rm GKdim}(Q\otimes B)$.
\end{center}
Now let $W$ be a frame of $Q\otimes B$, then there exist finitely
many $v_1,\dots,v_m\in B$ such that $V_W:=\,_R\langle
1,v_1,\dots,v_m\rangle$ is a frame of $B$ and $W\subseteq Q\otimes
V_W$. Indeed, let $W=\,_Q\langle z_1,\dots,z_k\rangle$, then for
every $1\leq i\leq k$,
\begin{center}
$z_i=q_{i1}\otimes v_{i1}+\cdots +q_{im_i}\otimes v_{im_i}$,
\end{center}
with $q_{ij}\in Q$ and $v_{ij}\in B$; hence, $W\subseteq\,
_Q\langle 1\otimes v_{11},\dots,1\otimes v_{km_k}\rangle$ and we
have found elements $v_1,\dots,v_m\in B$ such that
\begin{center}
$W\subseteq\,_Q\langle 1\otimes 1,1\otimes v_1,\dots,1\otimes v_m
\rangle\subseteq Q\otimes V_W$,
\end{center}
with $V_W:=\,_R\langle 1,v_1,\dots,v_m\rangle$. This proves the
claimed. Therefore,
\begin{center}
${\rm GKdim}(Q\otimes B)=\sup_{W}\overline{\lim_{n\to
\infty}}\log_n\dim_Q W^n\leq \sup_{V_W}\overline{\lim_{n\to
\infty}}\log_n\dim_Q (Q\otimes
V_W)^n=\sup_{V_W}\overline{\lim_{n\to \infty}}\log_n\dim_Q
(Q\otimes V_W^n)=\sup_{V_W}\overline{\lim_{n\to \infty}}\log_n{\rm
rank} V_W^n\leq {\rm GKdim}(B)$.
\end{center}
\end{proof}
Observe that $B$ is finitely generated if and only if $B$ has a
generator frame. In fact, if $x_1,\dots,x_m$ generate $B$ as
$R$-algebra, then $V:=\,_R\langle 1,x_1,\dots,x_m\rangle$ is a
generator frame of $B$; the converse is trivial from Definition
\ref{definitionGKforrings}.

\begin{proposition}
Let $B$ be a $R$-algebra and $V$ be a frame that generates $B$.
Then,
\begin{equation*}
{\rm GKdim}(B)=\overline{\lim_{n\to \infty}}\log_n({\rm rank}V^n).
\end{equation*}
Moreover, this equality is independent of the generator frame.
\end{proposition}
\begin{proof}
Notice that $Q\otimes V$ generates $Q\otimes B$, so from
Proposition \ref{proposition17.3.8}
\begin{center}
${\rm GKdim}(B)={\rm GKdim}(Q\otimes B)=\overline{\lim_{n\to
\infty}}\log_n(\dim_Q (Q\otimes V)^n)=\overline{\lim_{n\to
\infty}}\log_n(\dim_Q (Q\otimes V^n))=\overline{\lim_{n\to
\infty}}\log_n{\rm rank} V^n$.
\end{center}

The second part is trivial since the proof was independent of the
chosen frame.
\end{proof}

We conclude this preliminary section recalling the definition of
the skew $PBW$ extensions that we will consider in the main
theorem of the next section. Skew $PBW$ extensions cover many
noncommutative rings and algebras coming from mathematical physics
(see \cite{Lezama-reyes}). In \cite{Reyes5} the Gelfand-Kirillov
dimension of skew $PBW$ extensions that are algebras over fields
was computed. In the next section this result will be extended to
the case of algebras over commutative domains, i.e., for arbitrary
rings that are skew $PBW$ extensions. Another generalization of
the Gelfand-Kirillov dimension for skew $PBW$ extensions was
studied in \cite{lezamalatorre}, interpreting these extensions as
finitely semi-graded rings. The computation presented here agree
with \cite{Reyes5} and \cite{lezamalatorre}, choosing properly the
ring of coefficients of the extension.

\begin{definition}[\cite{LezamaGallego}]\label{gpbwextension}
Let $R$ and $A$ be rings. We say that $A$ is a \textit{skew $PBW$
extension of $R$} $($also called a $\sigma-PBW$ extension of
$R$$)$ if the following conditions hold:
\begin{enumerate}
\item[\rm (i)]$R\subseteq A$.
\item[\rm (ii)]There exist finitely many elements $x_1,\dots ,x_n\in A$ such $A$ is a left $R$-free module with basis
\begin{center}
${\rm Mon}(A):= \{x^{\alpha}=x_1^{\alpha_1}\cdots
x_n^{\alpha_n}\mid \alpha=(\alpha_1,\dots ,\alpha_n)\in
\mathbb{N}^n\}$, with $\mathbb{N}:=\{0,1,2,\dots\}$.
\end{center}
The set ${\rm Mon}(A)$ is called the set of standard monomials of
$A$.
\item[\rm (iii)]For every $1\leq i\leq n$ and $r\in R-\{0\}$ there exists $c_{i,r}\in R-\{0\}$ such that
\begin{equation}\label{sigmadefinicion1}
x_ir-c_{i,r}x_i\in R.
\end{equation}
\item[\rm (iv)]For every $1\leq i,j\leq n$ there exists $c_{i,j}\in R-\{0\}$ such that
\begin{equation}\label{sigmadefinicion2}
x_jx_i-c_{i,j}x_ix_j\in R+Rx_1+\cdots +Rx_n.
\end{equation}
Under these conditions we will write $A:=\sigma(R)\langle
x_1,\dots ,x_n\rangle$.
\end{enumerate}
\end{definition}
Associated to a skew $PBW$ extension $A=\sigma(R)\langle x_1,\dots
,x_n\rangle$ there are $n$ injective endomorphisms
$\sigma_1,\dots,\sigma_n$ of $R$ and $\sigma_i$-derivations, as
the following proposition shows.

\begin{proposition}[\cite{LezamaGallego}, Proposition 3]\label{sigmadefinition}
Let $A$ be a skew $PBW$ extension of $R$. Then, for every $1\leq
i\leq n$, there exist an injective ring endomorphism
$\sigma_i:R\rightarrow R$ and a $\sigma_i$-derivation
$\delta_i:R\rightarrow R$ such that
\begin{center}
$x_ir=\sigma_i(r)x_i+\delta_i(r)$,
\end{center}
for each $r\in R$.
\end{proposition}
A particular case of skew $PBW$ extension is when all $\sigma_i$
are bijective and the constants $c_{ij}$ are invertible.

\bigskip

\begin{definition}[\cite{LezamaGallego}]\label{sigmapbwderivationtype}
Let $A$ be a skew $PBW$ extension. $A$ is bijective if $\sigma_i$
is bijective for every $1\leq i\leq n$ and $c_{i,j}$ is invertible
for any $1\leq i<j\leq n$.
\end{definition}

If $A=\sigma(R)\langle x_1,\dots,x_n\rangle$ is a skew $PBW$
extension of the ring $R$, then, as was observed in Proposition
\ref{sigmadefinition}, $A$ induces injective endomorphisms
$\sigma_k:R\to R$ and $\sigma_k$-derivations $\delta_k:R\to R$,
$1\leq k\leq n$. Moreover, from the Definition
\ref{gpbwextension}, there exists a unique finite set of constants
$c_{ij}, d_{ij}, a_{ij}^{(k)}\in R$, $c_{ij}\neq 0$, such that
\begin{equation}\label{equation1.2.1}
x_jx_i=c_{ij}x_ix_j+a_{ij}^{(1)}x_1+\cdots+a_{ij}^{(n)}x_n+d_{ij},
\ \text{for every}\  1\leq i<j\leq n.
\end{equation}

Many important algebras and rings coming from mathematical physics
are particular examples of skew $PBW$ extensions: Habitual ring of
polynomials in several variables, Weyl algebras, enveloping
algebras of finite dimensional Lie algebras, algebra of
$q$-differential operators, many important types of Ore algebras,
algebras of diffusion type, additive and multiplicative analogues
of the Weyl algebra, dispin algebra $\mathcal{U}(osp(1,2))$,
quantum algebra $\mathcal{U}'(so(3,K))$, Woronowicz algebra
$\mathcal{W}_{\nu}(\mathfrak{sl}(2,K))$, Manin algebra
$\mathcal{O}_q(M_2(K))$, coordinate algebra of the quantum group
$SL_q(2)$, $q$-Heisenberg algebra \textbf{H}$_n(q)$, Hayashi
algebra $W_q(J)$, differential operators on a quantum space
$D_{\textbf{q}}(S_{\textbf{q}})$, Witten's deformation of
$\mathcal{U}(\mathfrak{sl}(2,K))$, multiparameter Weyl algebra
$A_n^{Q,\Gamma}(K)$, quantum symplectic space
$\mathcal{O}_q(\mathfrak{sp}(K^{2n}))$, some quadratic algebras in
3 variables, some 3-dimensional skew polynomial algebras,
particular types of Sklyanin algebras, homogenized enveloping
algebra $\mathcal{A}(\mathcal{G})$, Sridharan enveloping algebra
of 3-dimensional Lie algebra $\mathcal{G}$, among many others. For
a precise definition of any of these rings and algebras see
\cite{Lezama-reyes}.


\section{Computation of the Gelfand-Kirillov dimension\\ for the principal algebraic constructions}

Next we will compute the Gelfand-Kirillov dimension for the most
important algebraic constructions. For this we will apply the
correspondent properties of the Gelfand-Kirillov dimension over
fields (see \cite{Krause}).

\begin{theorem}\label{theorem17.4.4}
Let $B$ be a $R$-algebra.
\begin{enumerate}
\item[\rm (i)]If $B$ is finitely generated, then
\begin{center}
${\rm GKdim}(B)=0$ if and only if ${\rm rank} B<\infty$.
\end{center}
Moreover, if $B$ is a domain with ${\rm GKdim}(B)<\infty$, then
$B$ is an Ore domain.
\item[\rm (ii)]${\rm GKdim}(B[x_1,\dots,x_m])={\rm GKdim}(B)+m$.
\item[\rm (iii)]For $m\geq 2$, ${\rm GKdim}(R\{x_1,\dots,x_m\})=\infty$.
\item[\rm (iv)]${\rm GKdim}(M_n(B))={\rm GKdim}(B)$.
\item[\rm (v)]If $I$ is a proper two-sided ideal of $B$, then ${\rm GKdim}(B/I)\leq {\rm
GKdim}(B)$.
\item[\rm (vi)]Let $C$ be a subalgebra of $B$. Then, ${\rm GKdim}(C)\leq {\rm
GKdim}(B)$. Moreover, if $C\subseteq Z(B)$ $($the center of $B$$)$
and $B$ is finitely generated as $C$-module, then ${\rm
GKdim}(C)={\rm GKdim}(B)$.
\item[\rm (vii)]Let $C$ be a $R$-algebra. Then,
\begin{center}
${\rm GKdim}(B\times C)=\max\{{\rm GKdim}(B),{\rm GKdim}(C)\}$.
\end{center}
\item[\rm (viii)]Let $C$ be a $R$-algebra. Then,
\begin{center}
$\max\{{\rm GKdim}(B),{\rm GKdim}(C)\}\leq {\rm GKdim}(B\otimes
C)\leq {\rm GKdim}(B)+{\rm GKdim}(C)$.
\end{center}
In addition, suppose that $C$ contains a finitely generated
subalgebra $C_0$ such that ${\rm GKdim}(C_0)={\rm GKdim}(C)$, then
\begin{center}
${\rm GKdim}(B\otimes C)={\rm GKdim}(B)+{\rm GKdim}(C)$.
\end{center}
\item[\rm (ix)]Let $S$ be a multiplicative system of $B$ consisting of central regular elements. Then,
\begin{center}
${\rm GKdim}(BS^{-1})={\rm GKdim}(B)$.
\end{center}
\item[\rm (x)]If $G$ is a finite group, then
\[
{\rm GKdim}(B[G])={\rm GKdim}(B).
\]
\item[\rm (xi)]Let $B$ be $\mathbb{N}$-filtered and locally finite, i.e., for every $p\in \mathbb{N}$, $F_p(B)$ is a finitely generated $R$-submodule
of $B$. If $Gr(B)$ is finitely generated, then
\[
{\rm GKdim}(Gr(B))={\rm GKdim}(B).
\]
\item[\rm (xii)]Suppose that $B$ has a generator frame $V$. If $\sigma$
is a $R$-linear automorphism of $B$ such that $\sigma(V)\subseteq
V$ and $\delta$ is a $R$-linear $\sigma$-derivation of $B$, then
\[
{\rm GKdim}(B[x;\sigma,\delta])={\rm GKdim}(B)+1.
\]
\item[\rm (xiii)]Suppose that $B$ has a generator frame $V$ and let $A=\sigma(B)\langle x_1,\dotsc,x_t\rangle$ be
a bijective skew $PBW$ extension of $B$. If for $1\leq i\leq t$,
$\sigma_i,\delta_i$ are $R$-linear and $\sigma_i(V)\subseteq V$,
then
\[
{\rm GKdim}(A)={\rm GKdim}(B)+t.
\]
\end{enumerate}
\end{theorem}
\begin{proof}
(i) Notice that $Q\otimes B$ is finitely generated, so ${\rm
GKdim}(Q\otimes B)=0$ if and only if $\dim_Q(Q\otimes B)<\infty$.
Since ${\rm rank}B=\dim_Q(Q\otimes B)$, we get
\begin{center}
${\rm GKdim}(B)=0$ if and only if ${\rm GKdim}(Q\otimes B)=0$ if
and only if ${\rm rank}B<\infty$.
\end{center}
For the second statement, we will show that $B$ is a right Ore
domain, the proof on the left side is similar. Suppose contrary
that there exist $0\neq s\in B$ and $b\in B$ such that $sB\cap
bB=0$. Since $B$ is a domain, the following sum is direct
\begin{center}
$bB+sbB+s^{2}bB+s^{3}bB+\cdots $.
\end{center}
Let $V$ be a frame of $B$, then $W_V:=\,_R\langle V,s,b\rangle$ is
a frame of $B$ and
\begin{center}
$W_V^{2n}\supseteq\, _R\langle V,s,b\rangle^nV^n\supseteq
bV^n+sbV^n+s^2bV^n+\cdots+s^{n-1}bV^n$
\end{center}
and since the last sum is direct, then
\begin{center}
${\rm GKdim}(B)\geq \sup_{W_V}\overline{\lim_{n\to
\infty}}\log_n{\rm rank} W_V^{2n}\geq \sup_{V}\overline{\lim_{n\to
\infty}}\log_n{n\rm rank} V^n=1+{\rm GKdim}(B)$,
\end{center}
but this is impossible since ${\rm GKdim}(B)<\infty$.

\noindent (ii) It is enough to consider the case $m=1$. We have
the isomorphism of $R$-algebras $B[x]\cong B\otimes R[x]$, whence
we have the following isomorphism of $Q$-algebras:
\begin{center}
$Q\otimes B[x]\cong Q\otimes (B\otimes R[x])\cong (Q\otimes
B)\otimes R[x]\cong (Q\otimes B)[x]$.
\end{center}
Therefore,
\begin{center}
${\rm GKdim}(B[x])={\rm GKdim}(Q\otimes B[x])={\rm
GKdim}((Q\otimes B)[x])={\rm GKdim}(Q\otimes B)+1={\rm
GKdim}(B)+1$.
\end{center}
(iii) We have the isomorphism of $Q$-algebras
\begin{align*}
Q\otimes (R\{x_1,\dots,x_m\}) & \cong Q\{x_1,\dots,x_m\}\\
q\otimes \sum_{\alpha}r_\alpha x^\alpha & \mapsto
\sum_{\alpha}(qr_\alpha)x^\alpha.
\end{align*}
Therefore,
\begin{center}
${\rm GKdim}(R\{x_1,\dots,x_m\})={\rm GKdim}(Q\otimes
(R\{x_1,\dots,x_m\}))={\rm GKdim}(Q\{x_1,\dots,x_m\})=\infty$.
\end{center}
(iv) We have the isomorphism of $R$-algebras $M_n(B)\cong B\otimes
M_n(R)$, whence we have the following isomorphism of $Q$-algebras:
\begin{center}
$Q\otimes M_n(B)\cong Q\otimes (B\otimes M_n(R))\cong (Q\otimes
B)\otimes M_n(R)\cong M_n(Q\otimes B)$.
\end{center}
Therefore,
\begin{center}
${\rm GKdim}(M_n(B))={\rm GKdim}(Q\otimes M_n(B))={\rm
GKdim}(M_n(Q\otimes B))={\rm GKdim}(Q\otimes B)={\rm GKdim}(B)$.
\end{center}
(v) Let $W$ be a frame of $B/I$, we can assume that $W=\,_R\langle
\overline{1},\overline{w_1},\dots,\overline{w_t}\rangle$, then
$V_W:=\,_R\langle 1, w_1,\dots,w_t\rangle$ is a frame of $B$, and
hence, $Q\otimes W$ is a frame of $Q\otimes (B/I)$ and $Q\otimes
V_W$ is a frame of $Q\otimes B$. Observe that for every $n\geq 0$,
\begin{center}
$W^n=\,_R\langle \overline{w_{i_1}}\cdots \overline{w_{i_n}}\mid
w_{i_j}\in \{1,w_1,\dots,w_t\}\rangle=\overline{V_W^n}$,

$Q\otimes W^n=Q\otimes \overline{V_W^n}$,

$\dim_Q(Q\otimes W^n)=\dim_Q(Q\otimes \overline{V_W^n})\leq
\dim_Q(Q\otimes V_W^n)$.
\end{center}
The inequality can by justified in the following way. Both
$Q$-vector spaces $Q\otimes \overline{V_W^n}$ and $Q\otimes V_W^n$
have finite dimension and we have the surjective homomorphism of
$Q$-vector spaces $Q\otimes V_W^n\to Q\otimes \overline{V_W^n}$,
$q\otimes z\mapsto q\otimes \overline{z}$, with $q\in Q$, $z\in
V_W^n$.

Therefore, ${\rm GKdim}(B/I)\leq {\rm GKdim}(B)$.

\noindent (vi) Since $Q$ is $R$-flat, then $Q\otimes C$ is a
$Q$-subalgebra of $Q\otimes B$, hence
\begin{center}
${\rm GKdim}(C)={\rm GKdim}(Q\otimes C)\leq {\rm GKdim}(Q\otimes B
)={\rm GKdim}(B)$.
\end{center}
For the second statement, $Q\otimes C$ is a $Q$-subalgebra of
$Q\otimes Z(B)\subseteq Z(Q\otimes B)$; moreover, if
$B=\,_C\langle b_1,\dots,b_t\rangle$, with $b_i\in B$, $1\leq
i\leq t$, then
\begin{center}
$Q\otimes B=\,_{Q\otimes C}\langle 1\otimes b_1,\dots,1\otimes
b_t\rangle$.
\end{center}
Therefore,
\begin{center}
${\rm GKdim}(C)={\rm GKdim}(Q\otimes C)={\rm GKdim}(Q\otimes
B)={\rm GKdim}(B)$.
\end{center}
(vii) We have the following isomorphism of $Q$-algebras
\begin{align*}
Q\otimes (B\times C)& \cong (Q\otimes B)\times (Q\otimes C)\\
q\otimes (b,c) & \mapsto (q\otimes b,q\otimes c).
\end{align*}
Hence,
\begin{center}
${\rm GKdim}(B\times C)={\rm GKdim}(Q\otimes (B\times C))={\rm
GKdim}((Q\otimes B)\times (Q\otimes C))=\max\{{\rm GKdim}(Q\otimes
B),{\rm GKdim}(Q\otimes C)\}=\max\{{\rm GKdim}(B),{\rm
GKdim}(C)\}$.
\end{center}
(viii) We have the following isomorphisms of $Q$-algebras
\begin{center}
$(Q\otimes B)\otimes (Q\otimes C)\cong Q\otimes (B\otimes
Q)\otimes C\cong (Q\otimes Q)\otimes (B\otimes C)\cong Q\otimes
(B\otimes C)$.
\end{center}
Therefore,
\begin{center}
${\rm GKdim}(B\otimes C)={\rm GKdim}(Q\otimes (B\otimes C))={\rm
GKdim}((Q\otimes B)\otimes (Q\otimes C))\leq {\rm GKdim}(Q\otimes
B)+{\rm GKdim}(Q\otimes C)={\rm GKdim}(B)+{\rm GKdim}(C)$;

$\max\{{\rm GKdim}(B),{\rm GKdim}(C)\}=\max\{{\rm GKdim}(Q\otimes
B),{\rm GKdim}(Q\otimes C)\}\leq {\rm GKdim}((Q\otimes B)\otimes
(Q\otimes C))={\rm GKdim}(Q\otimes (B\otimes C))={\rm
GKdim}(B\otimes C)$.
\end{center}
For the second part, if $C$ contains a finitely generated
subalgebra $C_0$ such that ${\rm GKdim}(C_0)={\rm GKdim}(C)$, then
$Q\otimes C$ contains the finitely generated $Q$-algebra $Q\otimes
C_0$, ${\rm GKdim}(Q\otimes C_0)={\rm GKdim}(Q\otimes C)$, and
since $Q\otimes Q\cong Q$, we obtain {\small
\begin{center}
${\rm GKdim}(B\otimes C)={\rm GKdim}(Q\otimes (B\otimes C))={\rm
GKdim}((Q\otimes B)\otimes (Q\otimes C))={\rm GKdim}(Q\otimes
B)+{\rm GKdim}(Q\otimes C)={\rm GKdim}(B)+{\rm GKdim}(C)$.
\end{center}}
\noindent (ix) Let $W:=\,_R\langle
\frac{w_1}{s_1},\dots,\frac{w_t}{s_t}\rangle$ be a frame of
$BS^{-1}$; taking a common denominator $s$ we can assume
$W=\,_R\langle \frac{w_1}{s},\dots,\frac{w_t}{s}\rangle$. Then
$sW\subseteq\, _R\langle w_1,\dots,w_t\rangle\subseteq B$ and
observe that $V_W:=\,_R\langle 1,s,w_1,\dots,w_t\rangle$ is a
frame of $B$. For every $n\geq 0$, $W^n\subseteq
V_W^n\frac{1}{s^n}$ and since $Q$ is $R$-flat, then $Q\otimes
W^n\subseteq Q\otimes V_W^n\frac{1}{s^n}\cong Q\otimes V_W^n$
(isomorphism of $Q$-vector spaces), so $\dim_Q(Q\otimes W^n)\leq
\dim_Q(Q\otimes V_W^n)$, i.e., ${\rm rank}(W^n)\leq {\rm
rank}(V_W^n)$, whence ${\rm GKdim}(BS^{-1})\leq {\rm GKdim}(B)$.
Since $B\subseteq BS^{-1}$, (vi) completes the proof.

\noindent (x) We have the following isomorphism of $R$-algebras
\begin{align*}
B\otimes R[G] & \cong B[G]\\
b\otimes (\sum_{g\in G}r_g g) & \mapsto \sum_{g\in G}(b r_g)\cdot
g
\end{align*}
and from this we obtain the isomorphism of $Q$-algebras
\begin{center}
$Q\otimes B[G]\cong Q\otimes (B\otimes R[G])\cong (Q\otimes
B)\otimes R[G]\cong (Q\otimes B)[G]$.
\end{center}
Therefore,
\begin{center}
${\rm GKdim}(B[G])={\rm GKdim}(Q\otimes B[G])={\rm
GKdim}((Q\otimes B)[G])={\rm GKdim}(Q\otimes B)={\rm GKdim}(B)$.
\end{center}
(xi) $Q\otimes B$ has an induced natural $\mathbb{N}$-filtration
given by $\{Q\otimes F_p(B)\}_{p\in \mathbb{N}}$, moreover,
$Q\otimes B$ is locally finite. Since $Gr(B)$ is finitely
generated, then $Q\otimes Gr(B)$ is finitely generated. We have
the following isomorphism of $Q$-algebras:
\begin{center}
$Q\otimes Gr(B)\cong Gr(Q\otimes B)$, $q\otimes
\overline{b_p}\mapsto \overline{q\otimes b_p}$, with $q\in Q$ and
$b_p\in F_p(B)$
\end{center}
(this isomorphism is induced by the isomorphisms $Q\otimes
\frac{F_p(B)}{F_{p-1}(B)}\cong \frac{Q\otimes F_p(B)}{Q\otimes
F_{p-1}(B)}$). Hence, $Gr(Q\otimes B)$ is finitely generated and
we have
\begin{center}
${\rm GKdim}(Gr(B))={\rm GKdim}(Q\otimes Gr(B))={\rm
GKdim}(Gr(Q\otimes B))={\rm GKdim}(Q\otimes B)={\rm GKdim}(B)$.
\end{center}
(xii) We know that  $Q\otimes V$ generates the $Q$-algebra
$Q\otimes B$. Observe that
\begin{center}
$Q\otimes (B[x;\sigma,\delta])\cong (Q\otimes
B)[x;\overline{\sigma},\overline{\delta}]$
\end{center}
is an isomorphism of $Q$-algebras, where
$\overline{\sigma},\overline{\delta}:Q\otimes B\to Q\otimes B$ are
defined in the following natural way,
$\overline{\sigma}:=i_Q\otimes \sigma$,
$\overline{\delta}:=i_Q\otimes \delta$. It is easy to check that
$\overline{\sigma}$ is a bijective homomorphism of $Q$-algebras,
$\overline{\delta}$ is a $Q$-linear $\overline{\sigma}$-derivation
and $\overline{\sigma}(Q\otimes V)\subseteq Q\otimes V$.
Therefore,
\begin{center}
${\rm GKdim}(B[x;\sigma,\delta])={\rm GKdim}((Q\otimes
B)[x;\overline{\sigma},\overline{\delta}])={\rm GKdim}(Q\otimes
B)+1={\rm GKdim}(B)+1$.
\end{center}
(xiii) Note first that $A$ is a $R$-algebra. Let $V:=\,_R\langle
1,v_1,\dots,v_{l}\rangle$ be a generator frame of $B$, then
$\{V^n\}_{n\geq 0 }$ is a locally finite $\mathbb{N}$-filtration
of $B$, in particular, $B=\bigcup_{n\geq 0}V^n$. Let $m:={\rm
 max}\{m_{11},\dots,m_{tl}\}\geq 1$ with $\delta_i(v_j)\in V^{m_{ij}}$,
$1\leq i\leq t$ and $1\leq j\leq l$. Then, $\delta_i(V)\subseteq
V^m$ for every $1\leq i\leq t$. In addition, $m$ can be chosen
such that all parameters in (\ref{equation1.2.1}) belong to $V^m$.
By induction we can show that for $n\geq 0$,
$\delta_i(V^n)\subseteq V^{n+m}$. In fact,
$\delta_i(V^0)=\delta_i(R)=0\subseteq V^m$; $\delta_i(V)\subseteq
V^m\subseteq V^{m+1}$; assume that $\delta_i(V^n)\subseteq
V^{n+m}$ and let $z\in V^n$ and $v\in V$, then
$\delta_i(zv)=\sigma_i(z)\delta_i(v)+\delta_i(z)v$, but
$\sigma_i(z)\in V^n$, whence $\delta_i(zv)\in V^{n+1+m}$. Thus,
$\delta_i(V^{n+1})\subseteq V^{n+1+m}$. From this we obtain in
particular that $\delta_i(V^m)\subseteq V^{2m}$. Since $V^m$ is
also a generator frame, then we can assume that the generator
frame $V$ satisfies $\delta_i(V)\subseteq V^2$ and all parameters
in (\ref{equation1.2.1}) belong to $V$. From this, in turn, we
conclude that this generator frame $V=\,_R\langle
1,v_1,\dots,v_{l}\rangle$ satisfies
\begin{center}
$\delta_i(V^n)\subseteq V^{n+1}$, for every $n\geq 0$.
\end{center}
Let $X:=\,_R\langle 1,x_1,\dotsc,x_{t} \rangle$, and for every
$n\geq 1$ let
\begin{center}
$X_n:=\,_R\langle x^\alpha\in Mon(A)\mid \deg(x^\alpha)\leq
n\rangle$ (see Definition \ref{gpbwextension}).
\end{center}
Note that for every $n\geq 1$,
\begin{center}
$X_n\subseteq X^n\subseteq V^{n-1}X_n$.
\end{center}
The first inclusion is trivial and the second can be proved by
induction. Indeed, $X^1=X=X_1=RX_1=V^{1-1}X_1$; assume that
$X^{n-1}\subseteq V^{n-2}X_{n-1}$ and let $z\in X^n$, we can
suppose that $z$ has the form $z=z_1\cdots z_n$ with $z_i\in
\{1,x_1,\dots,x_t\}$, $1\leq i\leq n$. If at least one $z_i$ is
equal $1$, then $z\in X^{n-1}$, and by induction $z\in
V^{n-2}X_{n-1}\subseteq V^{n-1}X_n$. Thus, we can suppose that
every $z_i\in \{x_1,\dots,x_t\}$. If $z\in Mon(A)$, then $z\in
X_n\subseteq V^{n-1}X_n$. Assume that $z\notin Mon(A)$, then at
least one factor of $z$ should be moved in order to represent $z$
through the $B$-basis $Mon(A)$ of $A$. But the maximum number of
permutations in order to do this is $\leq n-1$ (and this is true
for every factor to be moved). Notice that in every permutation
arise the parameters of $A$, and as was observed above, these
parameters belongs to $V$. Hence, once the factor is in the right
position, we can apply induction and get that $z$, represented in
the standard form through the basis $Mon(A)$, belongs to
$V^{n-1}X_n$.

$X_n$ is a free left $R$-module with
\begin{equation*}
\dim_R X_n=\sum_{i=0}^n\binom{t+i-1}{i},
\end{equation*}
and for $n\ggg 0$
\begin{equation*}
(n/n-1)^t\leq \dim_R X_n\leq (n+1)^t.
\end{equation*}
Now we can complete the proof dividing it in two steps. For this,
let $W:=V+X$, then $W$ is a generating frame of $A$.

\textit{Step 1}. ${\rm GKdim}(A)\leq {\rm GKdim}(B)+t$. We will
show first that for every $n\geq 0$,
\begin{center}
$W^n\subseteq V^{n}X^{n}$.
\end{center}
For $n=0$, $W^0=R=V^0X^0$; for $n=1$, $W^1=V+X\subseteq VX$.
Suppose that $W^n\subseteq V^{n}X^{n}$ and let $w\in W$ and $z\in
W^n$, then denoting by $\delta$ any of elements of
$\{\delta_1,\dots,\delta_t\}$, we get $wz\in
(V+X)V^{n}X^{n}\subseteq V^{n+1}X^n+XV^nX^n\subseteq
V^{n+1}X^{n+1}+(V^nX+\delta(V^n))X^n\subseteq
V^{n+1}X^{n+1}+V^nX^{n+1}+V^{n+1}X^n\subseteq V^{n+1}X^{n+1}$.
Thus, $W^{n+1}\subseteq V^{{n+1}}X^{{n+1}}$.

Hence, ${\rm rank}W^{n}\leq {\rm rank}V^{n}X^{n}\leq {\rm
rank}V^{n}V^{n-1}X_n={\rm rank}V^{2n-1}X_n\leq {\rm
rank}V^{2n}X_n$, but since $V^{2n}\subseteq B$ and the $R$-basis
of $X_n$ is conformed by standard monomials with $\dim_R X_n\leq
(n+1)^t$, then ${\rm rank}V^{2n}X_n\leq (n+1)^t{\rm rank}V^{2n}$.
In fact, let $l:=\dim_R X_n$ and
$\{x^{\alpha_1},\dots,x^{\alpha_l}\}$ be a $R$-basis of $X_n$,
then we have the $R$-isomorphism
\begin{center}
$V^{2n}\oplus \cdots \oplus V^{2n}\to V^{2n}X_n$

$(b_1,\dots,b_l)\mapsto b_1x^{\alpha_1}+\cdots+b_lx^{\alpha_l}$,
\end{center}
and hence ${\rm rank}V^{2n}X_n={\rm rank}(V^{2n}\oplus \cdots
\oplus V^{2n})=l{\rm rank}V^{2n}\leq (n+1)^t{\rm rank}V^{2n}$.

Therefore,
\begin{center}
${\rm GKdim}(A)=\overline{\lim_{n\to \infty}}\log_n({\rm rank}
W^{n})\leq \overline{\lim_{n\to \infty}}\log_n({\rm rank}
V^{2n}X_n)\leq \overline{\lim_{n\to
\infty}}\log_n((n+1)^t\dim_Q(Q\otimes
V^{2n}))=t+\overline{\lim_{n\to \infty}}\log_n(\dim_Q(Q\otimes
V^n))=t+\overline{\lim_{n\to \infty}}\log_n({\rm rank}V^n)=t+{\rm
GKdim}(B)$.
\end{center}
\textit{Step 2}. ${\rm GKdim}(A)\geq {\rm GKdim}(B)+t$. Observe
that for every $n\geq 0$,

\begin{center}
$V^{n}X^n\subseteq W^{2n}$.
\end{center}
In fact, $V^0X^0=R=W^0$ and for $n\geq 1$, $V^nX^n\subseteq
(V+X)^n(V+X)^n=W^{2n}$. Therefore, $W^{2n}\supseteq
V^{n}X^n\supseteq V^{n}X_n$, and as in the step 1, we get ${\rm
rank}W^{2n}\geq {\rm rank}V^{n}X_{n}\geq (n/n-1)^t{\rm
rank}V^{n}$, whence
\begin{center}
${\rm GKdim}(A)=\overline{\lim_{n\to \infty}}\log_n({\rm rank}
W^{2n})\geq \overline{\lim_{n\to \infty}}\log_n({\rm rank}
V^{n}X_n)\geq\overline{\lim_{n\to
\infty}}\log_n((n/n-1)^t\dim_Q(Q\otimes
V^n))=t+\overline{\lim_{n\to \infty}}\log_n(\dim_Q(Q\otimes
V^n))=t+\overline{\lim_{n\to \infty}}\log_n({\rm rank}V^n)=t+{\rm
GKdim}(B)$.
\end{center}
\end{proof}

\begin{remark}
Comparing with the classical Gelfand-Kirillov dimension over
fields (see \cite{Krause}), we have the followings remarks about
Definition \ref{definitionGKforrings}.

(i) If $V$ is a generator frame of $B$, then $\{V^n\}_{n\geq 0}$
is a $\mathbb{N}$-filtration of $B$. Note that ${\rm GKdim}(B)$
measures the asymptotic behavior of the sequence $\{{\rm rank}
V^n\}_{n\geq 0}$.

(ii) From Proposition \ref{proposition17.3.8} we get that
\begin{center}
${\rm GKdim}(B)=r$ if and only if $r\in \{0\}\cup \{1\}\cup
[2,\infty]$.
\end{center}
Indeed, let $r:={\rm GKdim}(B)={\rm GKdim}(Q\otimes B)$. Then, it
is well-known (see \cite{Krause}) that $r\in \{0\}\cup \{1\}\cup
[2,\infty]$. Conversely, suppose that $r$ is in this union, then
there exists a $Q$-algebra $A$ such that ${\rm GKdim}(A)=r$.
Notice that $A$ is a $R$-algebra. Let $X$ be a $Q$-basis of $A$
and let $B$ be the $R$-subalgebra of $A$ generated by $X$. We have
the surjective homomorphism of $Q$-algebras
\begin{center}
$Q\otimes B\xrightarrow{\alpha} A$, $q\otimes b\mapsto q\cdot b$,
with $q\in Q$ and $b\in B$.
\end{center}
Hence, $A\cong (Q\otimes B)/\ker(\alpha)$, so $r={\rm
GKdim}(A)\leq {\rm GKdim}(Q\otimes B)={\rm GKdim}(B)$. On the
other hand, since $Q$ is $R$-flat we have the injective
homomorphism of $Q$-algebras $Q\otimes B\hookrightarrow Q\otimes
A$, and moreover, $Q\otimes A\cong A$, with $q\otimes a\mapsto
q\cdot a$ (isomorphism of $Q$-algebras). Therefore, $Q\otimes B$
is a $Q$-subalgebra of $A$, and hence, ${\rm GKdim}(B)={\rm
GKdim}(Q\otimes B)\leq {\rm GKdim}(A)=r$. Thus, ${\rm
GKdim}(B)=r$.

(iii) If $B$ is finitely generated and commutative, then ${\rm
GKdim}(B)$ is a nonnegative integer. Indeed, this property is
well-known (see \cite{Krause}, Theorem 4.5) for commutative
finitely generated algebras over fields since in such situation
${\rm GKdim}={\rm cl.Kdim}$. Hence, since $Q\otimes B$ is finitely
generated and commutative, then ${\rm GKdim}(B)={\rm
GKdim}(Q\otimes B)$ is a nonnegative integer.
\end{remark}

\section{Gelfand-Kirillov dimension for modules}

Definition \ref{definitionGKforrings} can be extended to modules.
Let $M$ be a right $B$-module, then $M$ is a $R-B$-bimodule:
Indeed, for $r\in R$, $m\in M$ and $b\in B$ we define $r\cdot
m:=m\cdot (r\cdot 1_B)$, and from this it is easy to check that
$(r\cdot m )\cdot b=r\cdot (m\cdot b)$.

\begin{definition}
Let $B$ a $R$-algebra and $M$ be a right $B$-module. The
\textit{Gelfand-Kirillov dimension of $M$} is defined by
\begin{equation*}
{\rm GKdim}(M):=\sup_{V,F}\overline{\lim_{n\to \infty}}\log_n{\rm
rank} FV^n,
\end{equation*}
where $V$ varies over all frames of $B$ and $F$ over all finitely
generated $R$-submodules of $M$. In addition, ${\rm
GKdim}(0):=-\infty$.
\end{definition}
Notice that $FV^n$ is a finitely generated $R$-submodule of $M$:
This follows from the identity $(r\cdot m)\cdot (r'\cdot
b)=(rr')\cdot(m\cdot b)$, with $r,r'\in R$, $m\in M$ and $b\in B$.
Moreover, $Q\otimes M$ is a right module over the $Q$-algebra
$Q\otimes B$, with product
\begin{center}
$(q\otimes m)\cdot (q'\otimes b):=(qq')\otimes (m\cdot b)$, where
$q,q'\in Q$, $m\in M$ and $b\in B$.
\end{center}

\begin{proposition}\label{proposition17.4.7}
Let $B$ a $R$-algebra and $M$ be a right $B$-module. Then,
\begin{center}
${\rm GKdim}(M)={\rm GKdim}(Q\otimes M)$.
\end{center}
\end{proposition}
\begin{proof}
Let $V$ be a frame of $B$ and $F$ be a finitely generated
$R$-submodule of $M$, then $Q\otimes F$ is a finitely generated
$Q$-vector subspace of the right $Q\otimes B$-module $Q\otimes M$,
and $Q\otimes V$ is a frame of $Q\otimes B$. In addition,
$(Q\otimes F)(Q\otimes V)^n=Q\otimes FV^n$ is a finitely generated
$Q$-subspace of $Q\otimes M$. Therefore, {\small
\begin{center}
${\rm GKdim}(M)=\sup_{V,F}\overline{\lim_{n\to \infty}}\log_n{\rm
rank} FV^n=\sup_{V,F}\overline{\lim_{n\to
\infty}}\log_n(\dim_Q(Q\otimes
FV^n))=\sup_{V,F}\overline{\lim_{n\to
\infty}}\log_n(\dim_Q((Q\otimes F)(Q\otimes V)^n))\leq {\rm
GKdim}(Q\otimes M)$.
\end{center}}
Now let $W$ be a frame of $Q\otimes B$, we showed in the proof of
Proposition \ref{proposition17.3.8} that there exist finitely many
$v_1,\dots,v_m\in B$ such that $V_W:=\,_R\langle
1,v_1,\dots,v_m\rangle$ is a frame of $B$ and $W\subseteq Q\otimes
V_W$. Similarly, if $G$ is a $Q$-subspace of $Q\otimes M$ of
finite dimension, then there exist finitely many $m_1,\dots,m_t\in
M$ such that $F_G:=\,_R\langle m_1,\dots,m_t\rangle$ satisfies
$G\subseteq Q\otimes F_G$. Moreover, for every $n\geq 0$,
\begin{center}
$GW^n\subseteq (Q\otimes F_G)(Q\otimes V_W)^n=(Q\otimes
F_G)(Q\otimes V_W^n)=Q\otimes F_GV_W^n$.
\end{center}
Therefore,
\begin{center}
${\rm GKdim}(Q\otimes M)=\sup_{W,G}\overline{\lim_{n\to
\infty}}\log_n\dim_Q GW^n\leq \sup_{V_W,F_G}\overline{\lim_{n\to
\infty}}\log_n\dim_Q (Q\otimes
F_GV_W^n)=\sup_{V_W,F_G}\overline{\lim_{n\to \infty}}\log_n{\rm
rank} F_GV_W^n\leq {\rm GKdim}(M)$.
\end{center}
\end{proof}
In the next theorem we present some basic properties of the
Gelfand-Kirillov dimension of modules.

\begin{theorem}\label{theoremformodules}
Let $B$ a $R$-algebra and $M$ be a right $B$-module. Then,
\begin{enumerate}
\item[\rm (i)]${\rm GKdim}(B_B)={\rm GKdim}(B)$.
\item[\rm (ii)]${\rm GKdim}(M)\leq {\rm GKdim}(B)$.
\item[\rm (iii)]Let $0\to K\to M\ \to L\to 0$ be an exact sequence of right $B$-modules,
then
\begin{center}
${\rm GKdim}(M)\geq \max\{{\rm GKdim}(K),{\rm GKdim}(L)\}$.
\end{center}
\item[\rm (iv)]Let $I$ be a two-sided ideal of $B$ and $MI=0$, then
\begin{center}
${\rm GKdim}(M_B)={\rm GKdim}(M_{B/I})$.
\end{center}
\item[\rm (v)]${\rm GKdim}(\sum_{i=1}^n M_i)=\max\{{\rm
GKdim}(M_i)\}_{i=1}^n={\rm GKdim}(\bigoplus_{i=1}^n M_i)$.
\end{enumerate}
\end{theorem}
\begin{proof}
(i) let $V$ be a frame of $B$, then for every $n\geq 0$,
$V^n\subseteq V^{n+1}=VV^n$, hence
\begin{equation*}
{\rm GKdim}(B)=\sup_{V}\overline{\lim_{n\to \infty}}\log_n{\rm
rank} V^n\leq \sup_{V,V}\overline{\lim_{n\to \infty}}\log_n{\rm
rank} VV^{n}\leq {\rm GKdim}(B_B).
\end{equation*}
On the other hand, if $F$ is a finitely generated $R$-submodule of
$B$, then for every frame $V$ of $B$ we have that $W_V:=V+F$ is a
frame of $B$, moreover, for every $n$,
$W_V^{n+1}=(V+F)^{n+1}\supseteq FV^n$, so
\begin{equation*}
{\rm GKdim}(B_B)=\sup_{V,F}\overline{\lim_{n\to \infty}}\log_n{\rm
rank}FV^n\leq \sup_{W_V}\overline{\lim_{n\to \infty}}\log_n{\rm
rank}W_V^{n+1}\leq {\rm GKdim}(B).
\end{equation*}
(ii) From Proposition \ref{proposition17.4.7} and (i) we get
\begin{center}
${\rm GKdim}(M)={\rm GKdim}(Q\otimes M)\leq {\rm GKdim}(Q\otimes
B)={\rm GKdim}(B)$.
\end{center}
(iii) Since $Q$ is a $R$-flat module we get the following exact
sequence of $Q\otimes B$-modules $0\to Q\otimes K\to Q\otimes M\
\to Q\otimes L\to 0$, whence
\begin{center}
${\rm GKdim}(M)={\rm GKdim}(Q\otimes M)\geq \max\{{\rm
GKdim}(Q\otimes K),{\rm GKdim}(Q\otimes L)\}=\max\{{\rm
GKdim}(K),{\rm GKdim}(L)\}$.
\end{center}
(iv) As in the part (v) of Theorem \ref{theorem17.4.4}, let $W$ be
a frame of $B/I$, we can assume that $W=\,_R\langle
\overline{1},\overline{w_1},\dots,\overline{w_t}\rangle$, then
$V_W:=\,_R\langle 1, w_1,\dots,w_t\rangle$ is a frame of $B$. Let
$G$ be a finitely generated $R$-submodule of $M_{B/I}$, then
$F_G:=G$ is also a finitely generated $R$-submodule of $M_B$. For
every $n\geq 0$ we have $\dim_Q(Q\otimes GW^n)=\dim_Q(Q\otimes
G\overline{V_W^n})=\dim_Q(Q\otimes F_GV_W^n)$. The last equality
in this case can by justified in the following way. The $Q$-vector
spaces $Q\otimes G\overline{V_W^n}$ and $Q\otimes F_GV_W^n$ have
finite dimension and we have the following homomorphisms of
$Q$-vector spaces:
\begin{center}
$Q\otimes F_GV_W^n\to Q\otimes G\overline{V_W^n}$, $q\otimes
g\cdot z\mapsto q\otimes g\cdot \overline{z}$, with $q\in Q$,
$g\in F_G$ and $z\in V_W^n$,

\smallskip

$Q\otimes G\overline{V_W^n}\to Q\otimes F_GV_W^n$, $q\otimes
g\cdot \overline{z}\mapsto q\otimes g\cdot z$, with $q\in Q$,
$g\in G$ and $z\in V_W^n$.
\end{center}
The last homomorphism is well-defined since $MI=0$. It is clear
that the composes of these homomorphisms give the identities.
Hence,
\begin{center}
${\rm GKdim}(M_{B/I})=\sup_{W,G}\overline{\lim_{n\to
\infty}}\log_n{\rm rank} GW^n=\sup_{W,G}\overline{\lim_{n\to
\infty}}\log_n\dim_Q (Q\otimes GW^n)=
\sup_{V_W,F_G}\overline{\lim_{n\to \infty}}\log_n\dim_Q (Q\otimes
F_GV_W^n)=\sup_{V_W,F_G}\overline{\lim_{n\to \infty}}\log_n{\rm
rank} F_GV_W^n\leq {\rm GKdim}(M_B)$.
\end{center}
Conversely, let $V:=\,_R\langle 1, v_1,\dots,v_t\rangle$ be a
frame of $B$, then $W_V:=\,_R\langle \overline{1},
\overline{v_1},\dots,\overline{v_t}\rangle$ is a frame of $B/I$;
let $F$ be a finitely generated $R$-submodule of $M_B$, then
$G_F:=F$ is a finitely generated $R$-submodule of $M_{B/I}$. As
above, for every $n\geq 0$ we have $\dim_Q(Q\otimes FV^n)=
\dim_Q(Q\otimes G_FW_V^n)$. Hence,
\begin{center}
${\rm GKdim}(M_{B})=\sup_{V,F}\overline{\lim_{n\to
\infty}}\log_n{\rm rank} FV^n=\sup_{V,F}\overline{\lim_{n\to
\infty}}\log_n\dim_Q (Q\otimes FV^n)=
\sup_{W_V,G_F}\overline{\lim_{n\to \infty}}\log_n\dim_Q (Q\otimes
G_FW_V^n)=\sup_{W_V,G_F}\overline{\lim_{n\to \infty}}\log_n{\rm
rank} G_FW_V^n\leq {\rm GKdim}(M_{B/I})$.
\end{center}
Therefore, ${\rm GKdim}(M_{B})={\rm GKdim}(M_{B/I})$.

(v) The equalities can be proved tensoring by $Q$.
\end{proof}

\section{Gelfand-Kirillov transcendence degree}

In \cite{GK} was defined the Gelfand-Kirillov transcendence degree
for algebras over fields (see also \cite{Zhangtranscendence}).
This notion can be extended to algebras over commutative domains.

\begin{definition}
Let $B$ a $R$-algebra. The \textit{Gelfand-Kirillov transcendence
degree} of $B$ is defined by
\[
{\rm Tdeg}(B):=\sup_{V}\inf_{b}{\rm GKdim}(R[bV]),
\]
where $V$ ranges over all frames of $B$ and $b$ ranges over all
regular elements of $B$.
\end{definition}

\begin{theorem}\label{theorem4.2}
Let $B$ a $R$-algebra. Then,
\begin{enumerate}
\item[\rm (i)]${\rm Tdeg}(B)\leq {\rm GKdim}(B)$.
\item[\rm (ii)]If $B$ is commutative, then
${\rm Tdeg}(B)={\rm GKdim}(B)$.
\end{enumerate}
\end{theorem}
\begin{proof}
Since $R[bV]$ is a $R$-subalgebra of $B$, then ${\rm
GKdim}(R[bV])\leq {\rm GKdim}(B)$ for all $V$ and $b$, whence
\begin{center}
${\rm Tdeg}(B)\leq {\rm GKdim}(B)$.
\end{center}
If $B$ is commutative, then $Q\otimes B$ is commutative and it is
known that for commutative algebras over fields the equality
holds, therefore, ${\rm Tdeg}(Q\otimes B)={\rm GKdim}(Q\otimes B)$
(see \cite{Zhangtranscendence}, Proposition 2.2). But note that
${\rm Tdeg}(Q\otimes B)\leq {\rm Tdeg}(B)$. In fact,
\[
{\rm Tdeg}(Q\otimes B):=\sup_{W}\inf_{z}{\rm GKdim}(Q[zW])\leq
\sup_{W}\inf_{u}{\rm GKdim}(Q[uW]),
\]
where $W$ ranges over all frames of $Q\otimes B$, $z$ ranges over
all regular elements of $Q\otimes B$ and $u$ over all regular
elements of $Q\otimes B$ of the form $1\otimes b$, with $b$ any
regular element of $B$ (if $b$ is a regular element of $B$, then
$1\otimes b$ is a regular element of $Q\otimes B$: Indeed, this
follows from the fact that $Q\otimes B\cong BS_0^{-1}$, with
$S_0:=R-\{0\}$, $1\otimes b\mapsto \frac{b}{1}$). As we saw in the
proof of Proposition \ref{proposition17.3.8}, given a frame $W$
there exists a frame $V_W$ of $B$ such that $W\subseteq Q\otimes
V_W$, hence $uW\subseteq u(Q\otimes V_W)=Q\otimes bV_W\subseteq
Q\otimes R[bV_W]$, whence $Q[zW]\subseteq Q\otimes R[bV_W]$, and
from this
\[
{\rm Tdeg}(Q\otimes B)\leq \sup_{V_W}\inf_{b}{\rm GKdim}(Q\otimes
R[bV_W])=\sup_{V_W}\inf_{b}{\rm GKdim}(R[bV_W])\leq {\rm Tdeg}(B).
\]
This proves the claimed.
\end{proof}

\begin{remark}
(i) Taking $R=\mathbb{Z}$ in Definition \ref{definitionGKforrings}
we get the notion of Gelfand-Kirillov dimension for arbitrary
rings, and hence, all properties in Theorems \ref{theorem17.4.4},
\ref{theoremformodules} and \ref{theorem4.2} hold in this
particular situation.

(ii) The proof of Theorem 2.2 in \cite{LezamaHelbert3} can be easy
adapted to the case of algebras over commutative domains. Indeed,
we can replace vector subspaces over the field $K$ and its
dimension by finitely generated submodules and its rank over the
commutative domain $R$. Thus, Theorem 2.2 in \cite{LezamaHelbert3}
can be extended in the following way: Let $R$ be a commutative
domain and $A$ be a right Ore domain. If $A$ is a finitely
generated $R$-algebra such that ${\rm GKdim}(A)<{\rm
GKdim}(Z(A))+1$, then
\begin{center}
$Z(Q_r(A))=\{\frac{p}{q}\mid p,q\in Z(A), q\neq 0\}\cong Q(Z(A))$.
\end{center}
\end{remark}



\end{document}